\documentclass[reqno,final]{amsart}
\usepackage{natbib}  %Nature-like bibliography
\usepackage{fancyhdr} %Headers
\usepackage{color} %Color definition
\usepackage{hyperref} %Internal and external links
\usepackage{graphicx} %Graphics inclusion

\definecolor{aleacolor}{rgb}{0.16,0.59,0.78}

\hypersetup{ breaklinks, colorlinks=true, linkcolor=aleacolor,
urlcolor=aleacolor, citecolor=aleacolor}

\renewcommand{\cite}{\citet}

\theoremstyle{plain}
\newtheorem{theorem}{Theorem}[section]
\newtheorem{proposition}[theorem]{Proposition}
\newtheorem{lemma}[theorem]{Lemma}
\newtheorem{corollary}[theorem]{Corollary}

\theoremstyle{definition}

\theoremstyle{remark}

\makeatletter \@addtoreset{equation}{section} \makeatother

\usepackage{amsmath}
\usepackage{amsthm}
\usepackage{epsfig}
\usepackage{psfrag}
\usepackage{subfigure}
\usepackage[mathscr,mathcal]{eucal}
\usepackage{amssymb}
\usepackage{enumerate}

\def \N {\mathbb N}
\def \Z {\mathbb Z}

\def \R {\mathbb R}

\newcommand{\ind}[1]{1\!\!1_{\{#1\}}}

\def\phi{\varphi}

\def\be{\begin{equation}}
\def\ee{\end{equation}}

\def\bea{\begin{eqnarray}}
\def\eea{\end{eqnarray}}

\newcommand{\abs}[1]{\left|{#1}\right|}

\newcommand{\prob}[2][]{\ensuremath{\mathbf{P}_{#1}\left(#2\right)}}
\newcommand{\expect}[2][]{\ensuremath{\mathbf{E}_{#1}\left(#2\right)}}

\newcommand{\condprob}[3][]{\ensuremath{\mathbf{P}_{#1}\left(#2\bigm|#3\right)}}
\newcommand{\Condprob}[3][]{\ensuremath{\mathbf{P}_{#1}\left(#2\biggm|#3\right)}}
\newcommand{\condexpect}[3][]{\ensuremath{\mathbf{E}_{#1}\left(#2\bigm|#3\right)}}

\def \well {\widetilde{\ell}}
\def \wX {\widetilde{X}}
\def\wT {\widetilde{T}}
\def\wLambda {\widetilde{\Lambda}}

\def\cF{\mathcal{F}}

\def\cT{\mathcal{T}}

\def\cX{\mathcal{X}}

\renewcommand{\d}{\,\mathrm d}

\newcommand{\BMl}{\mathfrak l}
\newcommand{\BMr}{\mathfrak r}

\renewcommand{\d}{\mathrm d}

\newcommand{\AxAsh}{{\lfloor Ax\rfloor,\lfloor\sqrt A\sigma h\rfloor}}

\title{Continuous time `true' self-avoiding random walk on $\Z$}

\author{ B\'alint T\'oth \and B\'alint Vet\H o}

\begin{document}

\maketitle

\begin{center}
\vspace*{-3ex}
Institute of Mathematics\\
Budapest University of Technology (BME)
\end{center}

\vspace*{3ex}

\begin{abstract}

We consider the continuous time version of the `true' or `myopic' self-avoiding
random walk with site repulsion in $1d$. The Ray\,--\,Knight-type method which
was applied in \citep{toth_95} to the discrete time and edge repulsion case is
applicable to this model with some modifications. We present a limit theorem
for the local time of the walk and a local limit theorem for the displacement.

\end{abstract}

\section{Introduction}

\subsection{Historical background}

Let $X(t)$, $t\in\Z_+:=\{0,1,2,\dots\}$ be a nearest neighbour walk on the
integer lattice $\Z$ starting from $X(0)=0$ and denote by $\ell(t,x)$,
$(t,x)\in\Z_+\times\Z$, its local time (that is: its occupation time measure)
on sites:
\[\ell(t,x):=\#\{0\le s\le t: X(s)=x\}\]
where $\#\{\dots\}$ denotes cardinality of the set. The true self-avoiding
random walk with site repulsion (STSAW) was introduced in
\citep{amit_parisi_peliti_83} as an example for a non-trivial random walk with
long memory which behaves qualitatively differently from the usual diffusive
behaviour of random walks. It is governed by the  evolution rules
\begin{align}
\condprob{X(t+1)=x\pm1} {{\cF}_t, \ X(t)=x} &= \frac {e^{-\beta\ell(t,x\pm1)}}
{e^{-\beta\ell(t,x+1)}+e^{-\beta\ell(t,x-1)}}\notag\\
&=\frac{e^{-\beta(\ell(t,x\pm1)-\ell(t,x))}}
{e^{-\beta(\ell(t,x+1)-\ell(t,x))}+e^{-\beta(\ell(t,x-1)-\ell(t,x))}},
\label{transprobstsaw}\\
\ell(t+1,x) &= \ell(t,x) + \ind{X(t+1)=x}.\notag
\end{align}

The extension of this definition to arbitrary dimensions is straightforward. In
\citep{amit_parisi_peliti_83}, actually, the multidimensional version of the
walk was defined. Non-rigorous -- nevertheless rather convincing -- scaling and
renormalization group arguments suggested that:
\begin{enumerate}
\item
In three and more dimensions, the walk behaves diffusively with a Gaussian
scaling limit of $t^{-1/2}X(t)$ as $t\to\infty$. See e.g.\
\citep{amit_parisi_peliti_83}, \citep{obukhov_peliti_83} and
\citep{horvath_toth_veto_10}.
\item
In one dimension (that is: the case formally defined above), the walk is
superdiffusive with a non-degenerate scaling limit of $t^{-2/3}X(t)$ as
$t\to\infty$, but with no hint about the limiting distribution.  See
\citep{peliti_pietronero_87}, \citep{toth_99} and \citep{toth_veto_08}.
\item
The critical dimension is $d=2$ where the Gaussian scaling limit is obtained
with logarithmic multiplicative corrections added to the diffusive scaling. See
\citep{amit_parisi_peliti_83} and \citep{obukhov_peliti_83}.
\end{enumerate}

These questions are still open. However, the scaling limit in one dimension of
a closely related object was clarified in \citep{toth_95}. The true
self-avoiding walk with self-repulsion defined in terms of the local times on
edges rather than sites is defined as follows:

Let $\wX(t)$, $t\in\Z_+:=\{0,1,2,\dots\}$ be yet again a nearest neighbour walk
on  the integer lattice $\Z$ starting from $\wX(0)=0$ and denote now by
$\well_{\pm}(t,x)$, $(t,x)\in\Z_+\times\Z$, its local time (that is: occupation
time measure) on unoriented edges:
\begin{align*}
\well_+(t,x)
&:=
\#\{0\le s < t: \{\wX(s),\wX(s+1)\}=\{x,x+1\} \},
\\
\well_-(t,x)
&:=
\#\{0\le s < t: \{\wX(s),\wX(s+1)\}=\{x,x-1\} \}.
\end{align*}
Note that $\well_+(t,x)=\well_-(t,x+1)$. The true self-avoiding random
walk with edge repulsion (ETSAW) is governed by the evolution rules
\begin{align*}
\condprob{\wX(t+1)=x\pm1} {{\cF}_t,\ \wX(t)=x} &= \frac
{e^{-2\beta\well_\pm(t,x)}}
{e^{-2\beta\well_+(t,x)}+e^{-2\beta\well_-(t,x)}}\\
&=\frac{e^{-\beta(\well_\pm(t,x)-\well_\mp(t,x))}}
{e^{-\beta(\well_+(t,x)-\well_-(t,x))}+
e^{-\beta(\well_-(t,x)-\well_+(t,x))}}\\
\well_\pm(t+1,x) &= \well_\pm(t,x) + \ind{\{\wX(t),\wX(t+1)\}=\{x,x\pm1\}}.
\end{align*}

In {\citep{toth_95}}, a limit theorem was proved for $t^{-2/3}\wX(t)$, as
$t\to\infty$. Later, in \citep{toth_werner_98}, a  space-time continuous
process $\R_+\ni t\mapsto \cX(t)\in\R$ was constructed -- called the true
self-repelling motion (TSRM) -- which possessed all the analytic and stochastic
properties of an assumed scaling limit of $\R_+\ni t\mapsto \cX^{(A)}(t):=
A^{-2/3}\wX([At])\in\R$. The invariance principle for this model has been
clarified in \citep{newman_ravishankar_06}.

A key point in the proof of \citep{toth_95}\ is a kind of Ray\,--\,Knight-type
argument which works for the ETSAW but not for the STSAW. (For the original
idea of Ray\,--\,Knight theory, see \citep{knight_63} and \citep{ray_63}.) Let
\[\wT_{\pm,x,h}:=\min\{t\ge0: \well_\pm(t,x)\ge h\},\qquad x\in\Z, \quad h\in\Z_+\]
be the so called inverse local times and
\[\wLambda_{\pm,x,h}(y):=\well_\pm(\wT_{\pm,x,h},y),\qquad x,y\in\Z, \quad h\in\Z_+\]
the local time sequence of the walk stopped at the inverse local times. It
turns out that, in the ETSAW case, for any fixed $(x,h)\in\Z\times\Z_+$, the
process $\Z\ni y\mapsto\wLambda_{\pm,x,h}(y)\in\Z_+$ is Markovian and it can
be thoroughly analyzed.

It is a fact that the similar reduction does not hold for the STSAW. Here, the
natural objects are actually slightly simpler to define:
\begin{align*}
T_{x,h}
&:=
\min\{t\ge0: \ell(t,x)\ge h\},&&
x\in\Z, & h\in\Z_+,
\\[1ex]
\Lambda_{x,h}(y)
&:=
\ell(T_{x,h},y),&&
x,y\in\Z, & h\in\Z_+.
\end{align*}
The process $\Z\ni y\mapsto\Lambda_{x,h}(y)\in\Z_+$ (with fixed
$(x,h)\in\Z\times\Z_+$) is not Markovian and thus the Ray\,--\,Knight-type of
approach fails. Nevertheless, this method works also for the model treated in
the present paper.

The main ideas of this paper are similar to those of \citep{toth_95}, but there
are essential differences, too. Those parts of the proofs which are the same as
in \citep{toth_95} will not be spelled out explicitly. E.g.\ the full proof of
Theorem \ref{thmtoth} is omitted altogether. We put the emphasis on those
arguments which differ genuinely from \citep{toth_95}. In particular, we
present some new coupling arguments.

This paper is organised as follows. First, we describe the model which we will
study and present our theorems. In Section \ref{RK}, we give the proof of
Theorem \ref{thmlimLambda} in three steps: we introduce the main technical
tools, i.e.\ some auxiliary Markov processes. Then we state technical lemmas
which are all devoted to check the conditions of Theorem \ref{thmtoth} cited
from \citep{toth_95}. Finally, we complete the proof using the lemmas. The
proof of these lemmas are postponed until Section \ref{proofs}. The proof of
Theorem \ref{thmXconv} is in Section \ref{Xconv}.

\subsection{The random walk considered and the main results}

Now, we define a version of true self-avoiding random walk in continuous time,
for which the Ray\,--\,Knight-type method sketched in the previous section is
applicable. Let $X(t)$, $t\in\R_+$ be a \emph{continuous time} random walk on
$\Z$ starting from $X(0)=0$ and having right continuous paths. Denote by
$\ell(t,x)$, $(t,x)\in\R_+\times\Z$ its local time (occupation time measure) on
sites:
\[\ell(t,x):=\abs{\{s\in[0,t)\,:\, X(s)=x\}}\]
where $|\{\dots\}|$ now denotes Lebesgue measure of the set indicated. Let
$w:\R\to(0,\infty)$ be an almost arbitrary rate function. We assume that it is
non-decreasing and not constant.

The law of the random walk is governed by the following jump rates and
differential equations (for the local time increase):
\begin{align}
\condprob{X(t+\d t)=x\pm1} {{\cF}_t, \ X(t)=x} &=
w(\ell(t,x)-\ell(t,x\pm1))\,\d t + o(\d t),\label{Xtrans}\\
\dot{\ell}(t,x) &= \ind{X(t)=x}\label{ltrans}
\end{align}
with initial conditions
\[X(0)=0, \qquad \ell(0,x)=0.\]
The dot in \eqref{ltrans} denotes time derivative. Note that for the the choice
of exponential weight function $w(u)=\exp\{\beta u\}$. This means exactly that
conditionally on a jump occurring at the instant $t$, the random walker jumps
to right or left from its actual position with probabilities
$e^{-\beta\ell(t,x\pm1)}/(e^{-\beta\ell(t,x+1)}+e^{-\beta\ell(t,x-1)})$, just
like in \eqref{transprobstsaw}. It will turn out that in the long run the
holding times remain of order one.

Fix $j\in\Z$ and $r\in\R_+$. We consider the random walk $X(t)$ running from
$t=0$ up to the stopping time
\begin{equation}
T_{j,r}=\inf\{t\ge0:\ell(t,j)\ge r\},\label{defT}
\end{equation}
which is the inverse local time for our model. Define
\begin{equation}
\Lambda_{j,r}(k):=\ell(T_{j,r},k)\qquad k\in\Z\label{Lambdadef}
\end{equation}
the local time process of $X$ stopped at the inverse local time.

Let
\begin{align*}
\lambda_{j,r}&:=\inf\{k\in\Z:\Lambda_{j,r}(k)>0\},\\
\rho_{j,r}&:=\sup\{k\in\Z:\Lambda_{j,r}(k)>0\}.
\end{align*}

Fix $x\in\R$ and $h\in\R_+$. Consider the two-sided reflected Brownian motion
$W_{x,h}(y)$, $y\in\R$ with starting point $W_{x,h}(x)=h$. Define the times of
the first hitting of $0$ outside the interval $[0,x]$ or $[x,0]$ with
\begin{align*}
\BMl_{x,h}&:=\sup\{y<0\wedge x:W_{x,h}(y)=0\},\\
\BMr_{x,h}&:=\inf\{y>0\vee x:W_{x,h}(y)=0\}
\end{align*}
where $a\wedge b=\min(a,b)$, $a\vee b=\max(a,b)$, and let
\begin{equation}
\cT_{x,h}:=\int_{\BMl_{x,h}}^{\BMr_{x,h}} W_{x,h}(y)\,\d y.\label{defcT}
\end{equation}

The main result of this paper is
\begin{theorem}\label{thmlimLambda}
Let $x\in\R$ and $h\in\R_+$ be fixed. Then
\begin{align}
A^{-1}\lambda_\AxAsh&\Longrightarrow\BMl_{0\wedge x,h},\\
A^{-1}\rho_\AxAsh&\Longrightarrow\BMr_{0\vee x,h},
\end{align}
and
\begin{equation}
\begin{split}
\left(\frac{\Lambda_\AxAsh(\lfloor Ay\rfloor)}{\sigma\sqrt A},
\frac{\lambda_\AxAsh}A\le y\le\frac{\rho_\AxAsh}A\right)\hspace*{8em}\\
\Longrightarrow\left(W_{x,h}(y), \BMl_{0\wedge x,h}\le y\le\BMr_{0\vee
x,h}\right)
\end{split}
\end{equation}
as $A\to\infty$ where
$\sigma^2=\int_{-\infty}^\infty u^2\rho(\d u)\in(0,\infty)$ with $\rho$ defined
by \eqref{defrho} and \eqref{defW} later.
\end{theorem}

\begin{corollary}\label{corTlim}
For any $x\in\R$ and $h\ge0$,
\begin{equation}
\frac{T_\AxAsh} {\sigma A^{3/2}}\Longrightarrow\cT_{x,h}.
\end{equation}
\end{corollary}

For stating Theorem \ref{thmXconv}, we need some more definitions. It follows
from \eqref{defcT} that $\cT_{x,h}$ has an absolutely continuous distribution.
Let
\begin{equation}
\omega(t,x,h):=\frac\partial{\partial t}\,\prob{\cT_{x,h}<t}\label{defomega}
\end{equation}
be the density of the distribution of $\cT_{x,h}$. Define
\[\varphi(t,x):=\int_0^\infty \omega(t,x,h)\,\d h.\]

Theorem 2 of \citep{toth_95} gives that, for fixed $t>0$, $\varphi(t,\cdot)$ is
a density function, i.e.
\begin{equation}
\int_{-\infty}^\infty \varphi(t,x)\,\d x=1.\label{intfi}
\end{equation}
One could expect that $\varphi(t,\cdot)$ is the density of the limit
distribution of $X(At)/A^{2/3}$ as $A\to\infty$, but we prove a similar
statement for their Laplace transform. We denote by $\hat\varphi$ the Laplace
transforms of $\varphi$:
\begin{equation}
\hat\varphi(s,x):=s\int_0^\infty e^{-st}\varphi(t,x)\,\d t.\label{deffihat}
\end{equation}

\begin{theorem}\label{thmXconv}
Let $s\in\R_+$ be fixed and $\theta_{s/A}$ a random variable of exponential
distribution with mean $A/s$ which is independent of the random walk $X(t)$.
Then, for almost all $x\in\R$,
\begin{equation}
A^{2/3}\prob{X(\theta_{s/A})=\lfloor A^{2/3}x\rfloor}\to\hat\varphi(s,x)
\end{equation}
as $A\to\infty$.
\end{theorem}

From this local limit theorem, the integral limit theorem follows immediately:
\[\lim_{A\to\infty}\prob{A^{-2/3}X(\theta_{s/A})<x}=\int_{-\infty}^x
\hat\varphi(s,y)\,\d y.\]

\section{Ray\,--\,Knight construction}\label{RK}

The aim of this section is to give a random walk representation of the local
time sequence $\Lambda_{j,r}$. Therefore, we introduce auxiliary Markov
processes corresponding to each edge of $\Z$. The process corresponding to the
edge $e$ is defined in such a way that its value is the difference of local
times of $X(T_{j,r})$ on the two vertices adjacent to $e$ where $X(T_{j,r})$ is
the process $X(t)$ stopped at an inverse local time. It turns out that the
auxiliary Markov processes are independent. Hence, by induction, the sequence
of local times can be given as partial sums of independent auxiliary Markov
processes. The proof of Theorem \ref{thmlimLambda} relies exactly on this
observation.

\subsection{The basic construction}\label{basic}

Let
\begin{equation}
\tau(t,k):=\ell(t,k)+\ell(t,k+1)\label{deftau}
\end{equation}
be the local time spent on (the endpoints of) the edge $\langle k,k+1\rangle$,
$k\in\Z$, and
\begin{equation}
\theta(s,k):=\inf\{t\ge0\,:\,\tau(t,k)>s\}
\end{equation}
its inverse. Further on, define
\begin{align}
\xi_k(s)&:=\ell(\theta(s,k),k+1)-\ell(\theta(s,k),k),\\
\alpha_k(s)&:=\ind{X(\theta(s,k))=k+1}-\ind{X(\theta(s,k))=k}.
\end{align}

A crucial observation is that, for each $k\in\Z$,
$s\mapsto(\alpha_k(s),\xi_k(s))$ is a Markov process on the state space
$\{-1,+1\}\times\R$. The transition rules are
\begin{align}
\condprob{\alpha_k(t+\d t)=-\alpha_k(t)}{\cF_t}
&=w(\alpha_k(t)\xi_k(t))\,\d t + o(\d t),\label{lawalpha}\\
\dot{\xi_k}(t)&=\alpha_k(t),\label{lawxi}
\end{align}
with some initial state $(\alpha_k(0),\xi_k(0))$. Furthermore, these processes
are independent. In plain words:
\begin{enumerate}
\item
$\xi_k(t)$ is the difference of time spent by $\alpha_k$ in the states $+1$ and
$-1$, alternatively, the difference of time spent by the walker on the sites
$k+1$ and $k$;
\item
$\alpha_k(t)$ changes sign with rate $w(\alpha_k(t)\xi_k(t))$ since the walker
jumps between $k$ and $k+1$ with these rates.
\end{enumerate}

The common infinitesimal generator of these processes is
\[(Gf)(\pm1, u)=\pm f'(\pm1,u) + w(\pm u)\big(f(\mp1,u)-f(\pm1,u)\big)\]
where $f'(\pm1,u)$ is the derivative with respect to the second variable. It is
an easy computation to check that these Markov processes are ergodic and their
common unique stationary measure is
\begin{equation}
\mu(\pm1,\d u)=\frac{1}{2Z}e^{-W(u)}\,\d u\label{defmu}
\end{equation}
where
\begin{equation}
W(u):=\int_0^u \left(w(v)-w(-v)\right)\d v\quad\text{and}\quad
Z:=\int_{-\infty}^\infty e^{-W(v)}\,\d v.\label{defW}
\end{equation}
Mind that, due to the condition imposed on $w$ (non-decreasing and
non-constant),
\begin{equation}
\lim_{\abs{u}\to\infty}\frac{W(u)}{\abs{u}}=\lim_{v\to\infty}(w(v)-w(-v))>0,
\label{Z<infty}
\end{equation}
and thus $Z<\infty$ and $\mu(\pm1,\d u)$ is indeed a probability measure on
$\{-1,+1\}\times\R$.

Let
\begin{equation}
\beta_\pm(t,k):=\inf\left\{s\ge0:\int_0^s \ind{\alpha_k(u)=\pm1}\,\d u\ge
t\right\}
\end{equation}
be the inverse local times of $(\alpha_k(t),\xi_k(t))$. With the use of them,
we can define the processes
\begin{equation}\label{defetak}
\eta_{k,-}(t):=\xi_k(\beta_-(t,k)),\qquad\eta_{k,+}(t):=-\xi_k(\beta_+(t,k)).
\end{equation}
which are also Markovian. By symmetry, the processes with different sign have
the same law. The infinitesimal generator of $\eta_{k,\pm}$ is
\[(Hf)(u)=-f'(u)+w(u)\int_u^\infty e^{-\int_u^v w(s)\,\d s}w(v)(f(v)-f(u))\,\d v.\]

It is easy to see that the Markov processes $\eta_{k,\pm}$ are ergodic and
their common unique stationary distribution is
\begin{equation}
\rho(\d u):=\frac1Z e^{-W(u)}\,\d u\label{defrho}
\end{equation}
with the notations \eqref{defW}. The stationarity of $\mu$ is not surprising
after \eqref{defmu}, but a straightforward calculation yields it also.

The main point is the following
\begin{proposition}
\begin{enumerate}
\item The processes $s\mapsto(\alpha_k(s),\xi_k(s))$, $k\in\Z$ are
independent Markov process with the same law given in
\eqref{lawalpha}--\eqref{lawxi}. They start from the initial states
$\xi_k(0)=0$ and
\[\alpha_k(0)=\left\{
\begin{array}{rl}
+1\quad & \mbox{if}\quad k<0,\\
-1\quad & \mbox{if}\quad k\ge0.
\end{array}\right.\]

\item The processes $s\mapsto\eta_{k,\pm}(s)$, $k\in\Z$ are independent
Markov processes if we consider exactly one of $\eta_{k,+}$ and $\eta_{k,-}$
for each $k$. The initial distributions are
\begin{align}
\prob{\eta_{k,+}(0)\in A}&=\left\{
\begin{array}{ll} Q(0,A)\quad & \mbox{if}\quad k\ge0,\\ \ind{0\in A}\quad &
\mbox{if}\quad k<0,
\end{array}\right.\label{initial+}\\[2ex]
\prob{\eta_{k,-}(0)\in A}&=\left\{
\begin{array}{ll} \ind{0\in A}\quad & \mbox{if}\quad k\ge0,\\ Q(0,A)\quad &
\mbox{if}\quad k<0.
\end{array}\right.\label{initial-}
\end{align}
\end{enumerate}
\end{proposition}

\subsection{Technical lemmas}

The lemmas of this subsection descibe the behaviour of the auxiliary Markov
processes $\eta_{k,\pm}$. Since they all have the same law, we denote them by
$\eta$ to keep the notation simple, and it means that the statement is true for
all $\eta_{k,\pm}$.

Fix $b\in\R$. Define the stopping times
\begin{align}
\theta_+&:=\inf\{t>0:\eta(t)\ge b\},\\
\theta_-&:=\inf\{t>0:\eta(t)\le b\}.
\end{align}

In our lemmas, $\gamma$ will always be a positive constant, which is considered
as being a small exponent, and $C$ will be a finite constant considered as
being large. To simplify the notation, we will use the same letter for
constants at different points of our proof. The notation does not emphasizes,
but their value depend on $b$.

First, we estimate the exponential moments of $\theta_-$ and $\theta_+$.

\begin{lemma}\label{lemma-momgenf}
There are $\gamma>0$ and $C<\infty$ such that, for all $y\ge b$,
\begin{equation}
\condexpect{\exp(\gamma\theta_-)}{\eta(0)=y}\le\exp(C(y-b)).\label{esttheta-}
\end{equation}
\end{lemma}

\begin{lemma}\label{lemma+momgenf}
There exists $\gamma>0$ such that
\begin{equation}
\condexpect{\exp(\gamma\theta_+)}{\eta(0)=b}<\infty.\label{esttheta+}
\end{equation}
\end{lemma}

Denote by $P^t=e^{tH}$ the transition kernel of $\eta$. For any $x\in\R$,
define the probability measure
\[Q(x,\d y):=\left\{\begin{array}{lcl}\exp(-\int_x^y w(u)\,\d u) w(y)\,\d y
& \mbox{if} & y\ge x,\\ 0 & \mbox{if} & y<x,\end{array}\right.\] which is the
conditional distribution of the endpoint of a jump of $\eta$ provided that
$\eta$ jumps from $x$. We show that the Markov process $\eta$ converges
exponentially fast to its stationary distribution $\rho$ defined by
\eqref{defrho} if the initial distribution is $0$ with probability $1$ or
$Q(0,\cdot)$.

\begin{lemma}\label{lemmaexpconv}
There are $C<\infty$ and $\gamma>0$ such that
\begin{equation}
\left\|P^t(0,\cdot)-\rho\right\|<C\exp(-\gamma t)\label{expconv}
\end{equation}
and
\begin{equation}
\left\|Q(0,\cdot)P^t-\rho\right\|<C\exp(-\gamma t).\label{nuexpconv}
\end{equation}
\end{lemma}

We give a bound on the decay of the tails of $P^t(0,\cdot)$ and $Q(0,\cdot)P^t$
uniformly in $t$.

\begin{lemma}\label{lemmaunifexpbound}
There are constants $C<\infty$ and $\gamma>0$ such that
\begin{equation}
P^t(0,(x,\infty))\le Ce^{-\gamma x}
\end{equation}
and
\begin{equation}
Q(0,\cdot)P^t(0,(x,\infty))\le Ce^{-\gamma x}
\end{equation}
for all $x\ge0$ and for any $t>0$ uniformly, i.e.\ the value of $C$ and
$\gamma$ does not depend on $x$ and $t$.
\end{lemma}

We introduce some notation from \citep{toth_95} and cite a theorem, which will
be the main ingredient of our proof. Let $A>0$ be the scaling parameter, and
let
\[S_A(l)=S_A(0)+\sum_{j=1}^l \xi_A(j)\qquad l\in\N\]
be a discrete time random walk on $\R_+$ with the law
\[\condprob{\xi_A(l)\in\d x}{S_A(l-1)=y}=\pi_A(\d x,y,l)\]
for each $l\in\N$ with
\[\int_{-y}^\infty \pi_A(\d x,y,l)=1.\]
Define the following stopping time of the random walk $S_A(\cdot)$:
\[\omega_{[Ar]}=\inf\{l\ge[Ar]:S_A(l)=0\}.\]

We give the following theorem without proof, because this is the continuous
analog of Theorem 4 in \citep{toth_95} and its proof is essentially identical
to that of the corresponding statement in \citep{toth_95}.
\begin{theorem}\label{thmtoth}
Suppose that the following conditions hold:
\begin{enumerate}
\item The step distributions $\pi_A(\cdot,y,l)$ converge exponentially fast as
$y\to\infty$ to a common asymptotic distribution $\pi$. That is, for each
$l\in\Z$,
\[\int_\R |\pi_A(\d x,y,l)-\pi(\d x)|<Ce^{-\gamma y}.\]

\item The asymptotic distribution is symmetric: $\pi(-\d x)=\pi(\d x)$, and its
moments are finite, in particular, denote
\begin{equation}
\sigma^2:=\int_\R x^2\pi(\d x).\label{defsigma}
\end{equation}

\item Uniform decay of the step distributions: for each $l\in\Z$,
\[\pi_A((x,\infty),y,l)\le Ce^{-\gamma x}.\]

\item Uniform non-trapping condition: The random walk is not trapped in a
bounded domain or in a domain away from the origin. That is, there is
$\delta>0$ such that
\begin{equation}
\int_\delta^\infty \pi_A(\d x,y,l)>\delta\quad\mbox{or}\quad
\int_{x=\delta}^\infty \int_{z=-\infty}^\infty \pi_A(\d x-z,y+z,l+1)\pi_A(\d
z,y,l)>\delta \label{nontrapping}
\end{equation}
and
\[\int_{-\infty}^{-(\delta\wedge y)} \pi_A(\d x,y,l)>\delta.\]
\end{enumerate}
Under these conditions, if
\[\frac{S_A(0)}{\sigma\sqrt A}\to h,\]
then
\begin{equation}
\left(\frac{\omega_{[Ar]}}A,\frac{S_A([Ay])}{\sigma\sqrt A}:0\le y\le
\frac{\omega_{[Ar]}}A\right)\Longrightarrow\left(\omega_r^W,|W_y|:0\le
y\le\omega_r^W\bigm||W_0|=h\right)
\end{equation}
in $\R_+\times D[0,\infty)$ as $A\to\infty$ where
\[\omega_r^W=\inf\{s>0:W_s=0\}\]
with a standard Brownian motion $W$ and $\sigma$ is given by \eqref{defsigma}.
\end{theorem}

\subsection{Proof of Theorem \ref{thmlimLambda}}

Using the auxiliary Markov processes introduced in Subsection \ref{basic}, we
can build up the local time sequence as a random walk. This
Ray\,--\,Knight-type construction is the main idea of the following proof.

\begin{proof}[Proof of Theorem \ref{thmlimLambda}]
Fix $j\in\Z$ and $r\in\R_+$. Using the definition \eqref{Lambdadef} and the
construction of $\eta_{k,\pm}$ \eqref{deftau}--\eqref{defetak}, we can
formulate the following recursion for $\Lambda_{j,r}$:
\begin{equation}\begin{aligned}
&\Lambda_{j,r}(j)=r,\\
&\Lambda_{j,r}(k+1)=\Lambda_{j,r}(k)+\eta_{k,-}(\Lambda_{j,r}(k)) \qquad &
\mbox{if}\quad k\ge j,\\
&\Lambda_{j,r}(k-1)=\Lambda_{j,r}(k)+\eta_{k-1,+}(\Lambda_{j,r}(k)) \qquad &
\mbox{if}\quad k\le j.
\end{aligned}\label{Lambdawalk}\end{equation}

It means that the processes $(\Lambda_{j,r}(j-k))_{k=0}^\infty$ and
$(\Lambda_{j,r}(j+k))_{k=0}^\infty$ are random walks on $\R_+$, they start from
$\Lambda_{j,r}(j)=r$, and the distribution of the following step always depends
on the actual position of the walker. In order to apply Theorem \ref{thmtoth},
we rewrite \eqref{Lambdawalk}:
\begin{align*}
\Lambda_{j,r}(j+k)&=h+\sum_{i=0}^{k-1} \eta_{j+i,-}(\Lambda_{j,r}(j+i)) &
k&=0,1,2,\dots,\\
\Lambda_{j,r}(j-k)&=h+\sum_{i=0}^{k-1} \eta_{j-i-1,+}(\Lambda_{j,r}(j-i)) &
k&=0,1,2,\dots.
\end{align*}
The step distributions of this random walks are
\[\pi_A(\d x,y,l)=\left\{\begin{array}{l} P^y(0,\d x)\\[1ex]
Q(0,\cdot)P^y(\d x)\end{array}\right.\]
according to \eqref{initial+}--\eqref{initial-}.

The exponential closeness of the step distribution to the stationary
distribution is shown by Lemma \ref{lemmaexpconv}. One can see from
\eqref{defrho} and \eqref{defW} that the distribution $\rho$ is symmetric and
it has a non-zero finite variance. Lemma \ref{lemmaunifexpbound} gives a
uniform exponential bound on the tail of the distributions $P^t(0,\cdot)$ and
$Q(0,\cdot)P^t$.

Since we only consider $[\lambda_{j,r},\rho_{j,r}]$, that is, the time
interval until $\Lambda_{j,r}$ hits $0$, we can force the walk to jump to
$1$ in the next step after hitting $0$, which does not influence our
investigations. It means that $\pi_A(\{1\},0,l)=1$ for $l\in\Z$, and with
this, the non-trapping condition \eqref{nontrapping} fulfils. Therefore,
Theorem \ref{thmtoth} is applicable for the forward and the backward walks,
and Theorem \ref{thmlimLambda} is proved.
\end{proof}

\section{The position of the random walker}\label{Xconv}

We turn to the proof of Theorem \ref{thmXconv}. First, we introduce the
rescaled distribution
\[\varphi_A(t,x):=A^{2/3}\prob{X(t)=\lfloor A^{2/3}x\rfloor}\]
where $t,x\in\R_+$. We define the Laplace transform of $\varphi_A$ with
\begin{equation}
\hat\varphi_A(s,x)=s\int_0^\infty e^{-st}\varphi_A(t,x)\,\d t,\label{deffiAhat}
\end{equation}
which is the position of the random walker at an independent random time of
exponential distribution with mean $A/s$.

We denote by $\hat\omega$ the Laplace transforms of $\omega$ defined in
\eqref{defomega} and rewrite \eqref{deffihat}:
\begin{gather*}
\hat\omega(s,x,h):=s\int_0^\infty e^{-st}\omega(t,x,h)\,\d t
=s\,\expect{e^{-s\cT_{x,h}}},\\
\hat\varphi(s,x)=s\int_0^\infty e^{-st}\varphi(t,x)\,\d t
=\int_0^\infty\hat\omega(s,x,h)\,\d h.
\end{gather*}
Note that the scaling relations
\begin{align}
\alpha\omega(\alpha t,\alpha^{2/3}x,\alpha^{1/3}h)&=\omega(t,x,h),\notag\\
\alpha^{2/3}\hat\varphi(\alpha^{-1}s,\alpha^{2/3}x)&=\hat\varphi(s,x)\label{fiscaling}
\end{align}
hold because of the scaling property of the Brownian motion.

\begin{proof}[Proof of Theorem \ref{thmXconv}]
The first observation for the proof is the identity
\begin{equation}
\prob{X(t)=k}=\int_{h=0}^\infty \prob{T_{k,h}\in(t,t+\d h)},\label{Ptransf}
\end{equation}
which follows from \eqref{defT}. If we insert it to the definition of
$\hat\varphi_A$ \eqref{deffiAhat}, then we get
\begin{equation}\label{fihatcomp}\begin{split}
\hat\varphi_A(s,x)&=sA^{-1/3}\int_0^\infty e^{-st/A}\prob{X(t)=\lfloor
A^{2/3}x\rfloor}\d t\\
&=sA^{-1/3}\int_0^\infty e^{-st/A}\int_{h=0}^\infty\prob{T_{\lfloor
A^{2/3}x\rfloor,h}\in(t,t+\d h)}\d t\\
&=sA^{-1/3}\int_0^\infty \expect{e^{-sT_{\lfloor A^{2/3}x\rfloor,h}/A}}\d h
\end{split}\end{equation}
using \eqref{Ptransf}. Defining
\[\hat\omega_A(s,x,h)=s\expect{\exp(-sT_{\lfloor A^{2/3}x\rfloor,\lfloor
A^{1/3}\sigma h\rfloor}/(\sigma A))}\]
gives us
\begin{equation}\label{fihatfinal}
\hat\varphi_A(s,x)=\int_0^\infty \hat\omega_A(\sigma s,x,h)\,\d h
\end{equation}
from \eqref{fihatcomp}. From Corollary \ref{corTlim}, it follows that, for any
$s>0$, $x\ge0$ and $h>0$,
\[\hat\omega_A(s,x,h)\to\hat\omega(s,x,h).\]

Applying Fatou's lemma in \eqref{fihatfinal}, one gets
\begin{equation}
\liminf_{A\to\infty}\hat\varphi_A(s,x) \ge\int_0^\infty \hat\omega(\sigma
s,x,h)\,\d h =\sigma^{2/3}\hat\varphi(s,\sigma^{2/3}x),\label{liminffihat}
\end{equation}
where we used \eqref{fiscaling} in the last equation. A consequence of
\eqref{intfi}, \eqref{liminffihat} integrated and a second application of
Fatou's lemma yield
\[1=\int_{-\infty}^\infty \hat\varphi(s,x)\,\d x \le \int_{-\infty}^\infty
\liminf_{A\to\infty}\hat\varphi_A(s,x)\,\d x \le \liminf_{A\to\infty}
\int_{-\infty}^\infty \hat\varphi_A(s,x)\,\d x=1,\] which gives that, for fixed
$s\in\R_+$, $\hat\varphi_A(s,x)\to\hat\varphi(s,x)$ holds for almost all
$x\in\R$, indeed.
\end{proof}

\section{Proof of lemmas}\label{proofs}

\subsection{Exponential moments of the return times}

\begin{proof}[Proof of Lemma \ref{lemma-momgenf}]
Consider the Markov process $\zeta(t)$ which decreases with constant speed $1$,
it has upwards jumps with homogeneous rate $w(-b)$, and the distribution of
the size of a jump is the same as that of $\eta$, provided that the jump
starts from $b$. In other words, the infinitesimal generator of $\zeta$ is
\[(Zf)(u)=-f'(u)+w(-b)\int_0^\infty e^{-\int_0^v w(b+s)\,\d s}w(b+v)
(f(u+v)-f(u))\,\d v.\]

Note that, by the monotonicity of $w$, $\eta$ and $\zeta$ can be coupled in such
a way that they start from the same position and, as long as $\eta\ge b$ holds,
$\zeta\ge\eta$ is true almost surely. It means that it suffices to prove
\eqref{esttheta-} with
\begin{equation}
\theta'_-:=\inf\{t>0:\zeta(t)\le b\}\label{theta'-def}
\end{equation}
instead of $\theta_-$. But the transitions of $\zeta$ are homogeneous in space,
which yields that \eqref{esttheta-} follows from the finiteness of
\begin{equation}
\condexpect{\exp(\gamma\theta'_-)}{\zeta(0)=b+1}.\label{theta'-}
\end{equation}

In addition to this, $\zeta$ is a supermartingale with stationary increments,
which gives us
\[\expect{\zeta(t)}=b+1-ct\]
with some $c>0$, if the initial condition is $\zeta(0)=b+1$. For
$\alpha\in\left(-\infty,\lim_{u\to\infty}\frac{W(u)}u\right)$
(c.f.\ \eqref{Z<infty}), the expectation
\[\log\expect{e^{\alpha(\zeta(t)-\zeta(0))}}\]
is finite, and negative for some $\alpha>0$. Hence, the martingale
\begin{equation}
M(t)=\exp\left(\alpha(\zeta(t)-\zeta(0))-t\log\expect{e^{\alpha(\zeta(1)-\zeta(0))}}\right)
\end{equation}
stopped at $\theta'_-$ gives that the expectation in \eqref{theta'-} is finite
with $\gamma=-\log\expect{e^{\alpha(\zeta(1)-\zeta(0))}}$.
\end{proof}

\begin{proof}[Proof of Lemma \ref{lemma+momgenf}]
First, we prove for negative $b$, more precisely, for which $w(-b)>w(b)$. In
this case, define the homogeneous process $\kappa$ with $\kappa(0)=b$ and generator
\[Kf(u)=-f'(u)+w(-b)\int_0^\infty e^{-w(b)s}w(b)(f(u+s)-f(u))\,\d s.\]
It is easy to see that there is a coupling of $\eta$ and $\kappa$, for which
$\eta\ge\kappa$ as long as $\eta\le b$. Therefore, it is enough to show
\eqref{esttheta+} with
\[\theta'_+:=\inf\{t>0:\kappa(t)\ge b\}\]
instead of $\theta_+$.

But $\kappa$ is a submartingale with stationary increments, for which
\[\log\expect{e^{\alpha(\kappa(t)-\kappa(0))}}\]
is finite if $\alpha\in(-\infty,w(b))$, and negative for some $\alpha<0$.
The statement follows from the same idea as in the proof of Lemma
\ref{lemma-momgenf}.

Now, we prove the lemma for the remaining case. Fix $b$, for which we already
know \eqref{esttheta+}, and chose $b_1>b$ arbitrarily. We start $\eta$ from
$\eta(0)=b_1$, and we decompose its trajectory into independent excursions
above and below $b$, alternatingly. Let
\begin{equation}
Y_0:=\inf\{t\ge0:\eta(t)\le b\},
\end{equation}
and by induction, define
\begin{align}
X_k&:=\inf\left\{t>0:\eta\left(\sum_{j=1}^{k-1} X_j+\sum_{j=0}^{k-1}
Y_j+t\right)\ge b\right\},
\label{defX}\\
Y_k&:=\inf\left\{t\ge0:\eta\left(\sum_{j=1}^k X_j+\sum_{j=0}^{k-1} Y_j+t
\right)\le b\right\} \label{defY}
\end{align}
if $k=1,2,\dots$. Note that $(X_k,Y_k)_{k=1,2,\dots}$ is an i.i.d.\ sequence of
pairs of random variables. Finally, let
\begin{equation}
Z_k:=X_k+Y_k\qquad k=1,2,\dots.\label{defZ}
\end{equation}

With this definition, the $Z_k$'s are the lengths of the epochs in a renewal
process. Lemma \ref{lemma-momgenf} tells us that $Y_0$ has finite exponential
moment. The same holds for $X_1,X_2,\dots$ because of the first part of this
proof for the case of small $b$. Note that the distribution of the upper
endpoint of a jump of $\eta$ conditionally given that $\eta$ jumps above $b$
is exactly $Q(b,\cdot)$. Since $Q(b,\cdot)$ decays exponentially fast, we
can use Lemma \ref{lemma-momgenf} again to conclude that
$\expect{\exp(\gamma Y_k)}<\infty$ for $\gamma>0$ small enough. Define
\begin{equation}
\nu_t:=\max\left\{n\ge0:\sum_{k=1}^n Z_k\le t\right\}\label{defnut}
\end{equation}
in the usual way. The following decomposition is true:
\begin{equation}\begin{split}
&\prob{\frac{\sum_{k=1}^{\nu_t+1} Y_k}t<\varepsilon}\\
&\qquad\qquad\le\prob{\frac{\nu_t+1}t<\frac12\frac1{\expect{Z_1}}}
+\prob{\frac{\sum_{k=1}^{\nu_t+1} Y_k}t<\varepsilon,
\frac{\nu_t+1}t\ge\frac12\frac1{\expect{Z_1}}}.
\end{split}\label{X/tdecomp}\end{equation}

Lemma 4.1 of \citep{vandenberg_toth_91} gives a large deviation principle for
the renewal process $\nu_t$, hence
\begin{equation}
\prob{\frac{\nu_t+1}t<\frac12\frac1{\expect{Z_1}}}
\le\prob{\frac{\nu_t}t<\frac12\frac1{\expect{Z_1}}}<e^{-\gamma t}
\end{equation}
with some $\gamma>0$. For the second term on the right-hand side in
\eqref{X/tdecomp},
\begin{equation}\begin{split}
&\prob{\frac{\sum_{k=1}^{\nu_t+1} Y_k}t<\varepsilon,
\frac{\nu_t+1}t\ge\frac12\frac1{\expect{Z_1}}}\\
&\hspace*{7em} =\prob{\frac{\sum_{k=1}^{\nu_t+1} Y_k}{\nu_t+1}
<\varepsilon\frac t{\nu_t+1},
\frac{\nu_t+1}t\ge\frac12\frac1{\expect{Z_1}}}\\
&\hspace*{7em}\le\prob{\frac{\sum_{k=1}^{\nu_t+1}Y_k}{\nu_t+1}
<2\varepsilon\expect{Z_1},
\frac{\nu_t+1}t\ge\frac12\frac1{\expect{Z_1}}}\\
&\hspace*{7em}\le\max_{n\ge\frac12\frac1{\expect{Z_1}}t}
\prob{\frac{\sum_{k=1}^n Y_k}n<2\varepsilon\expect{Z_1}},
\end{split}\label{X/test}\end{equation}
which is exponentially small for some $\varepsilon>0$ by standard large
deviation theory, and the same holds for the probability estimated is
\eqref{X/tdecomp}, which means that $\eta$ spends at least $\varepsilon t$
time above $b$ with overwhelming probability.

The inequality
\[\condprob{\theta_+>t}{\eta(0)=b_1}
\le\prob{\sum_{k=1}^{\nu_t+1}Y_k<\varepsilon t}
+\Condprob{\theta_+>t}{\eta(0)=b_1,\sum_{k=1}^{\nu_t+1}Y_k>\varepsilon t}\]
is obvious. The first term on the right-hand side is exponentially small by
\eqref{X/tdecomp}--\eqref{X/test}. In order to bound the second term, denote
by $J(t)$ the number of jumps when $\eta(s)\ge b$. The condition
$\sum_{k=1}^{\nu_t+1}Y_k>\varepsilon t$ means that this is the case in an at
least $\varepsilon$ portion of $[0,t]$. The rate of these jumps are at least
$w(-b)$ by the monotonicity of $w$. Note that $J(t)$ dominates stochastically
a Poisson random variable $L(t)$ with mean $w(-b)t$. Hence,
\begin{equation}
\prob{J(t)<\frac12w(-b)t}\le\prob{L(t)<\frac12w(-b)t}<e^{-\gamma t}\label{theta+tail}
\end{equation}
for $t$ large enough with some $\gamma>0$ by a standard large deviation
estimate.

Note that $Q$ is also monotone in the sense that
\[\int_{b_1}^\infty Q(x_1,\d y)<\int_{b_1}^\infty Q(x_2,\d y)\]
if $x_1<x_2$. Therefore, a jump of $\eta$, which starts above $b$, exits
$(-\infty,b_1]$ with probability at least
\[r=\int_{b_1}^\infty Q(b,\d y)>0.\]
Finally,
\[\begin{split}
&\Condprob{\theta_+>t}{\eta(0)=b_1,\sum_{k=1}^{\nu_t+1} Y_k>\varepsilon t}\\
&\quad\le\prob{J(t)<\frac12w(-b)\varepsilon t}
+\Condprob{\theta_+>t}
{J(t)\ge\frac12w(-b)\varepsilon t,\eta(0)=b_1,\sum_{k=1}^{\nu_t+1} Y_k>\varepsilon t}\\
&\quad\le e^{-\gamma t}+(1-r)^{\frac12w(-b)\varepsilon t}
\end{split}\]
by \eqref{theta+tail}, which is an exponential decay, as required.
\end{proof}

\subsection{Exponential convergence to the stationarity}

\begin{proof}[Proof of Lemma \ref{lemmaexpconv}]
First, we prove \eqref{expconv}. We couple two copies of $\eta$, say $\eta_1$
and $\eta_2$. Suppose that
\[\eta_1(0)=0\qquad\mbox{and}\qquad\prob{\eta_2(0)\in A}=\rho(A).\]
Their distribution after time $t$ are obviously $P^t(0,\cdot)$ and $\rho$,
respectively. We use the standard coupling lemma to estimate their variation
distance:
\[\left\|P^t(0,\cdot)-\rho\right\|\le\prob{T>t}\]
where $T$ is the random time when the two processes merge.

Assume that $\eta_1=x_1$ and $\eta_2=x_2$ with fixed numbers $x_1,x_2\in\R$.
Then there is a coupling where the rate of merge is
\[c(x_1,x_2):=w(-x_1\vee x_2)\exp\left(-\int_{x_1\wedge x_2}^{x_1\vee x_2}
w(z)\,\d z\right).\]
Consider the interval $I_b=(-b,b)$ where $b$ will be chosen later
appropriately. If $\eta_1=x_1$ and $\eta_2=x_2$ where $x_1,x_2\in I_b$, then
for the rate of merge
\begin{equation}
c(x_1,x_2)\ge w(-b)\exp\left(-\int_{-b}^b w(z)\,\d z\right)=:\beta(b)>0
\label{rateofmerge}
\end{equation}
holds if $w(x)>0$ for all $x\in\R$.

Let $\vartheta$ be the time spent in $I_b$, more precisely,
\begin{align*}
\vartheta_i(t)&:=|\{0\le s\le t:\eta_i(s)\in I_b\}|\qquad i=1,2,\\
\vartheta_{12}(t)&:=|\{0\le s\le t:\eta_1(s)\in I_b,\eta_2(s)\in I_b\}|.
\end{align*}
The estimate
\[\prob{T>t}\le\prob{\vartheta_{12}(t)<\frac t2}
+\condprob{T>t}{\vartheta_{12}(t)\ge\frac t2}\]
is clearly true. Note that
\[\Condprob{T>t}{\vartheta_{12}(t)\ge\frac
t2}\le\exp\left(-\frac12\beta(b)t\right)\]
follows from \eqref{rateofmerge}.

By the inclusion relation
\begin{equation}
\left\{\vartheta_{12}(t)<\frac t2\right\}\subset \left\{\vartheta_1(t)<\frac34
t\right\}\cup\left\{\vartheta_2<\frac34 t\right\},\label{inclusion}
\end{equation}
it suffices to prove that the tails of $\prob{\vartheta_i(t)<\frac34 t}$
decay exponentially $i=1,2$, if $b$ is large enough.

We will show that
\begin{equation}
\prob{\frac{|\{0\le s\le t:\eta(s)<b\}|}t<\frac78}\le e^{-\gamma t}.\label{7/8}
\end{equation}
A similar statement can be proved for the time spent above $-b$, therefore
another inclusion relation like \eqref{inclusion} gives the lemma.

First, we verify that the first hitting of level $b$
\[\inf\{s>0:\eta_i(s)=b\}\]
has finite exponential moment, hence, it is negligible with overwhelming
probability and we can suppose that $\eta_i(0)=b$. Indeed, for any fixed
$\varepsilon>0$, the measures $\rho$ and $Q(b,\cdot)$ assign exponentially
small weight to the complement of the interval $[-\varepsilon t,\varepsilon t]$
as $t\to\infty$. From now on, we suppress the subscript of $\eta_i$, we forget
about the initial values, and assume only that
$\eta(0)\in[-\varepsilon t,\varepsilon t]$.

If $\eta(0)\in[b,\varepsilon t]$, then recall the proof Lemma
\ref{lemma-momgenf}. There, we could majorate $\eta$ with a homogeneous
process $\zeta$. If we define
\[a:=\condexpect{\theta'_-}{\zeta(0)=b+1}\]
with the notation \eqref{theta'-def}, which is finite by Lemma
\ref{lemma-momgenf}, then from a large deviation principle,
\begin{equation}
\condprob{\theta_-(t)>2a\varepsilon t}{\eta(0)\in[b,\varepsilon t]}
\le\condprob{\theta'_-(t)>2a\varepsilon t}{\eta(0)\in[b,\varepsilon t]}
<e^{-\gamma t}\label{theta-largedev}
\end{equation}
with some $\gamma>0$.

If $\eta(0)\in[-\varepsilon t,b]$, then we can neglect that piece of the
trajectory of $\eta$ which falls into the interval $[0,\theta_+]$, because
without this, $\vartheta(t)$ decreases and the bound on \eqref{7/8}
becomes stronger. Since $\eta$ jumps at $\theta_+$ a.s.\ and the distribution
of $\eta(\theta_+)$ is $Q(b,\cdot)$, we can use the previous observations
concerning the case $\eta(0)\in[b,\varepsilon t]$.

Using \eqref{theta-largedev}, it is enough to prove that
\[\prob{\frac{|\{0\le s\le t:\eta(s)<b\}|}t<\frac78+2a\varepsilon}\le e^{-\gamma t}\]
with the initial condition $\eta(0)=b$ where the value of $b$ is not specified
yet. We introduce $X_k,Y_k,Z_k$ and $\nu_t$ as in \eqref{defX}--\eqref{defnut}
with $Y_0\equiv0$. The only difference is that here we want to ensure a given
portion of time spent below $b$ with high probability with the appropriate
choice of $b$. With the same idea as in the proof of Lemma \ref{lemma+momgenf}
in \eqref{X/tdecomp}--\eqref{X/test}, we can show that
\[\prob{\frac{\sum_{k=1}^{\nu_t+1} X_k}t\le\frac78+2a\varepsilon}\]
is exponentially small by large deviation theory if we choose $b$ large enough
to set $\expect{X_1}/\expect{Z_1}$ (the expected portion of time spent below
$b$) sufficiently close to $1$. With this, the proof of \eqref{expconv} is
complete, that of \eqref{nuexpconv} is similar.
\end{proof}

\subsection{Decay of the transition kernel}

\begin{proof}[Proof of Lemma \ref{lemmaunifexpbound}]
We return to the idea that the partial sums of $Z_k$'s form a renewal process.
Remember the definitions \eqref{defX}--\eqref{defnut}. This proof relies on the
estimate
\[|\eta(t)|\le Z_{\nu_t+1},\]
which is true, because the process $\eta$ can decrease with speed at most $1$.
Therefore, it suffices to prove the exponential decay of the tail of
$Z_{\nu_t+1}$.

Define the \emph{renewal measure} with
\[U(A):=\sum_{n=0}^\infty \prob{\sum_{k=1}^n Z_k\in A}\]
for any $A\subset\R$. We consider the \emph{age} and the \emph{residual waiting
time}
\begin{align*}
A_t&:=t-\sum_{k=1}^{\nu_t} Z_k,\\
R_t&:=\sum_{k=1}^{\nu_t+1} Z_k-t
\end{align*}
separately. For the distribution of the former $H(t,x):=\prob{A_t>x}$, the
renewal equation
\begin{equation}
H(t,x)=(1-F(t))\ind{t>x}+\int_0^t H(t-s,x)\,\d F(s)\label{reneq}
\end{equation}
holds where $F(x)=\prob{Z_1<x}$. \eqref{reneq} can be deduced by conditioning
on the time of the first renewal, $Z_1$. From Theorem (4.8) in
\citep{durrett_95}, it follows that
\begin{equation}
H(t,x)=\int_0^t(1-F(t-s))\ind{t-s>x}U(\d s).\label{formH}
\end{equation}

As explained after \eqref{defZ}, Lemma \ref{lemma-momgenf} and Lemma
\ref{lemma+momgenf} with $b=0$ together imply that $1-F(x)\le C e^{-\gamma x}$
with some $C<\infty$ and $\gamma>0$. On the other hand,
\[U([k,k+1])\le U([0,1])\]
is true, because, in the worst case, there is a renewal at time $k$. Otherwise,
the distribution of renewals in $[k,k+1]$ can be obtained by shifting the
renewals in $[0,1]$ with $R_k$. We can see from \eqref{formH} by splitting the
integral into segments with unit length that
\[H(t,x)\le U([0,1])\sum_{k=\lfloor x\rfloor}^\infty C e^{-\gamma k},\]
which is uniform in $t>0$.

With the equation
\[\{R_t>x\}=\{A_{t+x}\ge x\}=\{\mbox{no renewal in } (t,t+x]\},\]
a similar uniform exponential bound can be deduced for the tail $\prob{R_t>x}$.
Since $Z_{\nu_t+1}=A_t+R_t$, the proof is complete.
\end{proof}

\subsection*{Acknowledgements}
This research was partially supported by the OTKA (Hungarian National Research
Fund) grants K 60708 (for B.\ T.\ and B.\ V.) and TS 49835 (for B.\ V.).
B.\ V.\ thanks the kind hospitality of the Erwin Schr\"odinger Institute
(Vienna) where part of this work was done.

\vfill

\hbox{ \phantom{M} \hskip7cm
\vbox{\hsize8cm {\noindent Address of authors:\\
{\sc
Institute of Mathematics\\
Budapest University of Technology \\
Egry J\'ozsef u.\ 1\\
H-1111 Budapest, Hungary}\\[10pt]
e-mail:\\
{\tt balint{@}math.bme.hu}\\
{\tt vetob{@}math.bme.hu}
}}}
\end{document}